\documentclass[final, 12pt]{article}

\usepackage{amsfonts}
\usepackage{makeidx}
\usepackage{amsmath}
\usepackage{amssymb}
\usepackage{amsthm}
\usepackage{bm}
\usepackage{cite}
\usepackage{showkeys}
\usepackage[english]{babel}
\usepackage{dblaccnt}
\usepackage{accents}
\usepackage{graphicx}
\usepackage{subfig}
\usepackage{psfrag}\usepackage{color}

\usepackage{amscd}
\usepackage[mathscr]{euscript}

\newcommand*{\N}{\ensuremath{\mathbb{N}}}
\newcommand*{\Z}{\ensuremath{\mathbb{Z}}}
\newcommand*{\R}{\ensuremath{\mathbb{R}}}
\newcommand*{\C}{\ensuremath{\mathbb{C}}}
\newtheorem{defi}{Definition}[section]
\newtheorem{lemma}[defi]{Lemma}
\newtheorem{theorem}[defi]{Theorem}

\newtheorem{assumption}[defi]{Assumption}
\renewcommand{\i}{\mathrm{i}}
\renewcommand{\d}[1]{\,\mathrm{d}#1 \,}

\newcommand{\ol}[1]{\overline{#1}}

\newcommand{\Lk}{\mathbb{L}_k}
\DeclareMathOperator{\supp}{\mathrm{supp}}

\newcommand{\grad}{\nabla}
\renewcommand{\div}{\mathrm{div} \,}

\renewcommand{\Re}{\mathrm{Re}\,}
\renewcommand{\Im}{\mathrm{Im}\,}

\newcommand{\eps}{\varepsilon}

\renewcommand{\d}[1]{\,\mathrm{d}#1 \,}

\renewcommand{\div}{\mathrm{div}}
\newcommand{\Rg}{\mathrm{Range}}

\begin{document}

\sloppy

\title{On the inverse scattering from anisotropic periodic layers and transmission 
eigenvalues\thanks{This work was partially supported by NSF grant DMS-1812693 and the Faculty Enhancement Program Award 
from Kansas State University}}
\author{Isaac Harris\thanks{Department of Mathematics, Purdue University, West Lafayette, IN 47907; (\texttt{harri814@purdue.edu})} \and Dinh-Liem Nguyen\thanks{Department of Mathematics, Kansas State University, Manhattan, KS 66506; (\texttt{dlnguyen@ksu.edu},
 \texttt{jtsands@ksu.edu},  \texttt{ttruong@ksu.edu}) } 
 \and Jonathan Sands\footnotemark[3] \and Trung Truong\footnotemark[3]
}
\date{}

\maketitle

\begin{abstract}
This paper is concerned with the inverse scattering and the transmission eigenvalues for  anisotropic periodic layers. 
For the inverse scattering problem, we study the Factorization method for  shape reconstruction of  the periodic  
layers from near field scattering data. This method provides a fast numerical algorithm as well as a 
unique determination for the shape reconstruction of the scatterer. We present a rigorous justification and 
numerical examples for the factorization method. The transmission eigenvalue problem in scattering have recently attracted 
a lot of attentions. Transmission eigenvalues can be determined from scattering data and they can provide information 
about the material parameters of the scatterers. In this paper we formulate the interior transmission eigenvalue problem
 and  prove the existence of infinitely many transmission eigenvalues for the scattering from anisotropic periodic layers. 
\end{abstract}

\textbf{Keywords.} inverse scattering, Factorization method, transmission eigenvalues, anisotropic periodic structure, 
  TM polarization

\bigskip

\textbf{AMS subject classification.} 35R30, 78A46, 65C20


\section{Introduction}
We study in this paper the inverse scattering problem and the transmission eigenvalues for anisotropic periodic 
structures in $\R^2$. The periodic structures of interest are supposed to be unboundedly periodic  in 
the horizontal direction and bounded in the vertical direction. This can be considered as the model for one-dimensional 
photonic crystals. 
We are mainly concerned with a sampling method for  shape reconstruction of the periodic scatterers from near field data,
the formulation of the transmission eigenvalue problem and existence of transmission eigenvalues. This study 
is motivated by applications of nondestructive evaluations for periodic structures in optics. The development of 
numerical methods for shape reconstruction of periodic structures in inverse scattering  has been an active research 
topic during the past years, see~\cite{Arens2005, Arens2003a, Bao2014, Elsch2012, Hadda2017, Jiang2017, 
Lechl2013b, Nguye2014, Sandf2010, Yang2012} for a non exhaustive list of results. 
However, most of the results focus on the case of isotropic periodic scattering structures. 
The case of anisotropic periodic structures has not been studied much.  The first part of this paper is devoted to 
a study of the Factorization method for solving the inverse scattering problem in two dimensions. This
two-dimensional problem can be considered as the (simplified) TM-polarization case of the full Maxwell problem for anisotropic  
periodic structures. This is an extension of
the results for the half space problem  in~\cite{Nguye2014} to the full space one. While we only need to measure
 scattering data  above the periodic layer in the half space problem, the analysis for the Factorization method for 
the full space case in this paper requires the data measured from both sides of the periodic layer. Therefore, the 
measurement operator and the analysis of its factorization have to be modified for the theoretical analysis. 
We want to point out that the inverse scattering problem  for anisotropic periodic layers has also been recently studied in the paper~\cite{Nguye2019}. 
The sampling methods developed in~\cite{Nguye2019} can detect the local perturbation and/or the  periodic layer itself.  
However, it is assumed in the cited paper that the complement of the periodic layer in one period is connected, while our theoretical 
analysis does not need this assumption. 

The interior transmission eigenvalues in scattering theory have recently received a great attention
thanks to their mathematical interests and applications. 
Transmission eigenvalue problems are non self-adjoint as well as non-linear which makes their investigation mathematically interesting. One can determine these transmission eigenvalues from scattering data (see for e.g. \cite{Cakon2010a} and \cite{Kirsc2013}). More importantly, they can provide information about the material parameters of the scattering medium. {\color{black} In general, they are monotone with respect to the material parameters which means they can be used as a target signature to determine changes in the scatterer. }
 The transmission eigenvalue problem for anisotropic medium scattering has been studied in~\cite{Harri2014, Kleef2018}. We also refer
 to~\cite{Cakon2015} for a study of homogenization  of the transmission eigenvalue problem for periodic media. 
Recent results and developments of the transmission eigenvalues and their applications can be found in~\cite{Cakon2016}.
The interior transmission eigenvalue problem is less well understood in the case of periodic media, 
and has not been studied in the context considered in this paper. We present first in this work a formulation 
of the  transmission eigenvalue problem. Second, we follow the theory in~\cite{Cakon2015} to prove  that 
there exists infinitely many  transmission eigenvalues
for the periodic layer scattering  under certain assumption.

The outline of the paper is as follows. After the introduction, we present in Section 2 a  formulation of the direct scattering problem
as well as a brief discussion on its variational form. Section 3 is dedicated to  the inverse scattering problem
and the justification of the Factorization method for solving the inverse problem. In Section 4, we formulate the transmission 
eigenvalue problem for the scattering from anisotropic periodic structures and prove the existence of infinitely many transmission 
eigenvalues.  We present some numerical examples in  Section 5 to demonstrate the performance of the Factorization method
for the shape reconstruction, and a short summary of the paper in Section 6.


\section{Direct problem formulation}
We consider a two-dimensional  layer which is $2\pi$-periodic in $x_1$-direction and bounded in $x_2$-direction. 
Let  $A$ be a matrix-valued bounded function which is  $2\pi$-periodic with respect to $x_1$. Suppose that 
 this periodic scattering layer  is fully characterized by $A$ and that the medium outside of the
 layer is homogeneous which means $A = I$ in this area.  Note that we could assume an 
 arbitrary value for the period of the layer. The period $2\pi$  is chosen for the convenience of the presentation.

Suppose that this periodic layer is illuminated by  the incident plane wave 
 \begin{equation}
 \label{eq:incidentWave}
 u_{\mathrm{in}}(x) = \text{e}^{\i (d_1x_1 + d_2x_2)}
 \end{equation}
 where  $(d_1,d_2)^\top$ is the wave propagation vector direction
 satisfying $d_1^2+d_2^2 = k^2$, $k>0$ is the wave number  and $d_2 \neq 0$. The latter condition means we 
 are only interested  in incident plane wave propagating downward or upward toward the layer, see also Figure~\ref{fg0}
 for a schematic of the periodic scattering. 
The scattering of this incident plane wave by  the anisotropic periodic  layer
produces the scattered field $u_{\mathrm{sc}}$ described by
\begin{align}
\label{eq:HmodeEquation1}
 \div(A\nabla u_{\mathrm{sc}}) + k^2 u_{\mathrm{sc}} = -\div(Q\nabla  u_{\mathrm{in}}) \quad \text{in } \R^2,
\end{align}
where $Q$ is the contrast given by
\[
 Q = A - I.
\]
It  is important to note that the incident field  $u_{\mathrm{in}}$ is $\alpha$-quasi-periodic in $x_1$ with period $2\pi$, that means, for
$\alpha := d_1$, it satisfies
\[
  u_{\mathrm{in}}(x_1 + 2\pi n,x_2) = \text{e}^{\i2\pi n \alpha} u_{\mathrm{in}}(x_1,x_2), \quad n\in \Z, \quad (x_1,x_2)^\top \in \R^2.
\]
From now on we call functions with this property quasi-periodic functions for short. 
It is well known for this  scattering problem that the scattered field $u_{\mathrm{sc}}$ 
must also be quasi-periodic (in $x_1$), and that the direct problem of finding 
the scattered field can be reduced to one period 
$$
\Omega := (-\pi,\pi)\times \R.
$$
Let $h>0$ be a positive constant such that
\begin{align}
\label{hh}
h>\sup \big\{|x_2|: \, (x_1,x_2)^\top \in \supp(Q) \big\},
\end{align}
where $\supp(Q)$ is the support of the contrast $Q$.
The direct scattering problem is completed by the Rayleigh 
expansion condition for the scattered field
\begin{equation}
\label{eq:radiationCondition}
  u_{\mathrm{sc}}(x) =
  \begin{cases}
   \sum_{n\in \Z} \widehat{u}^{+}_n \text{e}^{\i\alpha_n x_1 + \i\beta_n (x_2- h)}, &  x_2 > h, \\
      \sum_{n\in \Z} \widehat{u}^{-}_n \text{e}^{\i\alpha_n x_1 - \i\beta_n (x_2 + h)}, &  x_2 < - h, 
  \end{cases}
\end{equation}
where 
\begin{align*}
   \alpha_n := \alpha + n, \quad 
\beta_n:= \begin{cases} 
              \sqrt{k^2 - \alpha_n^2}, & k^2 \geq \alpha_n^2 \\ 
              \i \sqrt{\alpha_n^2- k^2}, & k^2 < \alpha_n^2
            \end{cases}, \quad  n \in \Z,
\end{align*}
and $(\widehat{u}^{\pm}_n)_{n\in\Z}$ are the Rayleigh sequences given by
$$ 
\widehat{u}^{\pm}_n := \frac{1}{2 \pi}\int_{0}^{2\pi} u^s(x_1,\pm h) e^{-\i \alpha_n x_1} \d{x_1}.
$$
The condition~\eqref{eq:radiationCondition} means that the scattered field $u_{\mathrm{sc}}$ is 
an outgoing wave (see e.g.~\cite{Bonne1994}). Note that only a finite number 
of terms in~\eqref{eq:radiationCondition} are propagating plane waves which 
are called propagating modes, the rest are evanescent modes which correspond 
to exponentially decaying terms. This also implies the absolute convergence of the series in
~\eqref{eq:radiationCondition}. From now, we call a function satisfying~\eqref{eq:radiationCondition} a
radiating function. In addition, we also assume that $\beta_n$ is nonzero for all $n$ which means  
the Wood anomalies are excluded in our analysis.

 \begin{figure}[!ht]
 \centering
  \includegraphics[width=12cm]{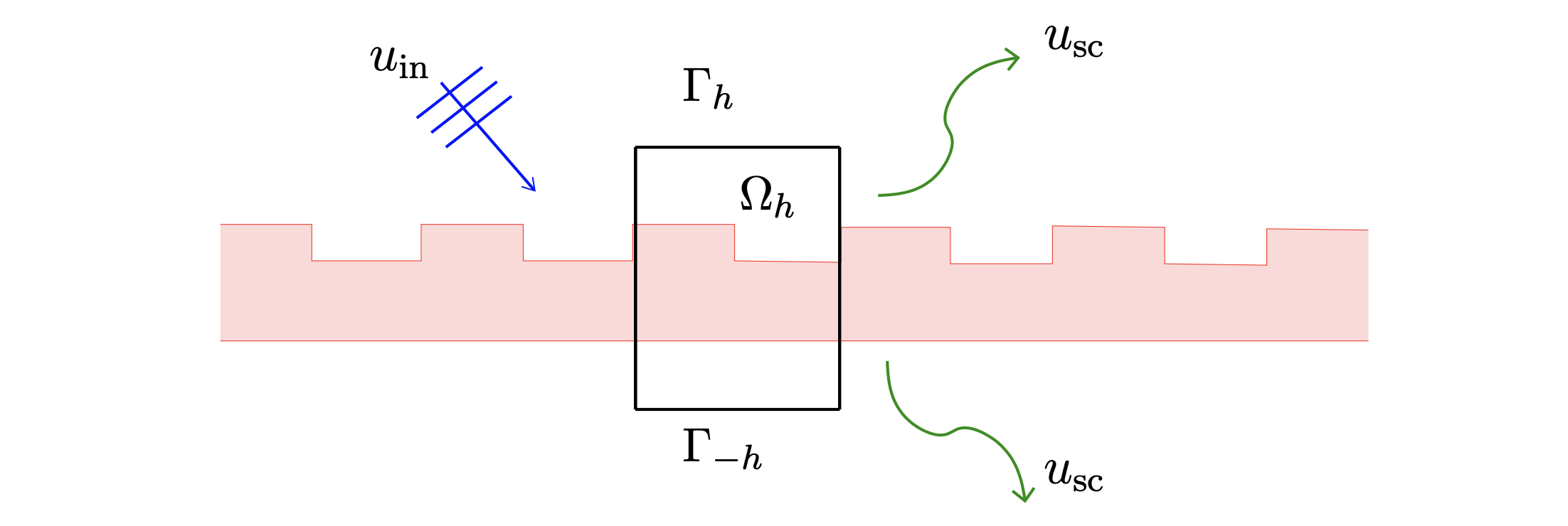}
 \caption{Schematic of the scattering  from a penetrable periodic layer}
  \label{fg0}
 \end{figure}

Well-posedness of scattering problem~\eqref{eq:HmodeEquation1}--\eqref{eq:radiationCondition}
is well-known, see for instance~\cite{Bonne1994}.
For $h$ given in~\eqref{hh}, consider a truncation of $\Omega$ as  
\[
  \Omega_ h:=(-\pi,\pi)\times(- h, h), \quad \Gamma_{\pm h}:= (-\pi,\pi)\times\{\pm h\}.
\]
The variational form of the direct problem is formulated in 
\[
H^1_{\alpha}(\Omega_h) := \big\{u\in H^1(\Omega_ h): \,  u=U|_{\Omega_ h}
  \text{ for some } \alpha\text{-quasi-periodic } U \in H^1_{\text{loc}}(\R^2) \big\}.
\]
Now the  variational problem is to find $u_{\mathrm{sc}} \in H^1_{\alpha}(\Omega_ h)$ 
such that, for $f =  Q \nabla u_{\mathrm{in}}$,
\begin{equation}
 \begin{split}
   \label{eq:variationalPeriodic}
   \mathcal{B}(u_{\mathrm{sc}},v;A)  = -\int_{\Omega_ h} f \cdot \nabla \overline{v} \d{x}, \quad \text{for all } v \in H^1_{\alpha}(\Omega_ h),
 \end{split}
\end{equation}
 where
$$
   \mathcal{B}(u_{\mathrm{sc}},v;A) := 
   \int_{\Omega_ h}  A\nabla u_{\mathrm{sc}}  \cdot \nabla \ol{v}
    - k^2  u_{\mathrm{sc}} \ol{v}  \d{x} 
   - \int_{\Gamma_ h} T^+(u_{\mathrm{sc}})\ol{v}\d{s}
    - \int_{\Gamma_{- h}} T^-(u_{\mathrm{sc}})\ol{v}\d{s}.
$$
 Here the operators 
$T^{\pm}: H^{1/2}(\Gamma_{\pm h}) \to H^{-1/2}(\Gamma_{\pm h})$, 
defined by $T^{\pm}(\varphi) = \i \sum_{n \in\Z} \beta_n \widehat{\varphi}^{\pm}_n e^{\i \alpha_n x_1}$,
are the exterior Dirichlet-to-Neumann operators. It was proved in~\cite{Bonne1994} 
that, under  uniform ellipticity conditions on $A$, the direct problem 
has a unique solution for all but a discrete set of wave number $k$. In this paper 
we only consider the wave number $k>0$ such that the direct problem has a unique solution.  
We also note that the solution to problem~\eqref{eq:variationalPeriodic} can be extended from $\Omega_h$ to 
$\R^2$ as the solution of the direct scattering problem by using the Rayleigh radiation 
condition in  $x_2$-direction and  the quasi-periodicity in  $x_1$-direction.


\section{The inverse problem}
In this section we formulate the inverse problem of interest.  First we define that
 \[
 \overline{D}  \subset \R^2: \text{ the support of the contrast $Q$ in the period $\Omega$}. 
 \]
 The following assumption is important for the analysis of the inverse problem.
 \begin{assumption}
 \label{th:assumption1}
 Suppose that $D$ is a Lipschitz domain
and  that $\Omega\setminus \ol{D}$ has at most two unbounded connected components.
The contrast $Q(x)$ is complex-valued and symmetric for almost all $x \in \R^2$. 
  There exists positive constants $c_1, c_2$ such that $\Re Q(x)\xi\cdot \ol{\xi} \geq c_1|\xi|^2$ and 
   $\Im Q(x)\xi\cdot \ol{\xi} \leq -c_2|\xi|^2$
   for all $\xi \in \C^2$ and almost all $x \in D$. 
  Furthermore, the well-defined square root $(\Re Q(x))^{1/2}$ is also positive definite with inverse $(\Re Q(x))^{-1/2}$, 
  and $(\Re Q)^{\pm1/2}$ belong to $L^{\infty}(D)^{2\times2}$.
 \end{assumption}

Note that by assuming that $\Im Q$ is negative definite we assume absorbing materials for the periodic scatterer.  This assumption 
also excludes the transmission eigenvalues, and is important in our factorization method analysis. However, for the study of 
transmission eigenvalues in Section 4 we assume that $\Im Q = 0$.
Denote by $\ell^2(\Z)$ the space of square summable sequences. Thanks to 
the well-posedness of the direct problem we can define the solution operator $G: [L^2(D)]^2 \to [\ell^2(\Z)]^2$ 
by  
\begin{align}
\label{G}
G(f) = (\widehat{u}^+_n, \widehat{u}^-_n)^\top_{n \in \Z},
\end{align}
where $(\widehat{u}^+_n, \widehat{u}^-_n)^\top_{n \in \Z}$ are the Rayleigh sequences of solution $u \in H^1_{\alpha}(\Omega_h)$ of
\begin{equation}
\label{eq:VariationalInverse2}
\mathcal{B}(u,v;A)
= -\int_{D}Q (\Re Q)^{-1/2} f\cdot \nabla \ol{v} \d{x} \quad \text{for all } v \in H^1_{\alpha}(\Omega_ h).
\end{equation}
For the inverse problem we consider the quasi-periodic incident  plane waves
\begin{equation}
\label{eq:IncidentFields}
\varphi^{\pm}_n = \text{e}^{\i(\alpha_nx_1 - \beta_nx_2)} \pm  \text{e}^{\i(\alpha_nx_1 + \beta_nx_2)}, \quad n \in \Z.
\end{equation}
Since problem~\eqref{eq:VariationalInverse2} is linear, a linear combination of several incident fields will lead to a 
corresponding linear combination of resulting scattered fields. We consider a linear combination 
using sequences $(a_n)_{n\in \Z} = (a^+_n,a^-_n)^\top_{n\in \Z} \in [\ell^2(\Z)]^2$ and define the  operator 
$H:[\ell^2(\Z)]^2 \to [L^2(D)]^2$ by
\begin{equation}
\label{eq:IncidentField2}
 H(a_n) = (\Re Q)^{1/2} \sum\limits_{n \in \Z} \left(\frac{a^+_n}{\beta_nw^+_n} \nabla \varphi^{+}_n 
+ \frac{a^-_n}{\beta_nw^-_n}\nabla \varphi^{-}_n\right), 
\end{equation}
where
\begin{equation*}
 w^+_n: =  \left \{ \begin{array}{ll}
 \i,  & k^2 > \alpha_n^2, \\
\text{e}^{-\i\beta_n h},   & k^2 < \alpha_n^2,\\
          \end{array} \right. \quad
w^-_n: =  \left \{ \begin{array}{ll}
 1,  & k^2 > \alpha_n^2, \\
\text{e}^{-\i\beta_n h},   & k^2 < \alpha_n^2.\\
          \end{array} \right. 
\end{equation*}
We use the weights $\beta_n w^{\pm}_n$ in the linear combination~\eqref{eq:IncidentField2} 
 for the convenience of our calculations. 

Motivated by applications in near field optics we consider near field measurements in our inverse problem. 
More precisely, we define  
the \textit{near field} operator $N: [\ell^2(\Z)]^2 \to [\ell^2(\Z)]^2$ 
 mapping  sequence $(a_n)_{n \in\Z}$ to the Rayleigh sequences of the scattered field generated by the 
linear combinations of the incident plane waves  in~\eqref{eq:IncidentFields}, i.e.
\begin{equation}
\label{NearField}
 N(a_n) := (\widehat{u}^+_j, \widehat{u}^-_j)^\top_{j\in\Z},
\end{equation}
where $u \in H^1_{\alpha}(\Omega_h)$ is the radiating solution to~\eqref{eq:VariationalInverse2} 
for  $f = H(a_n)$.  Here we note that the Rayleigh sequences in~\eqref{NearField} 
are given by the solution operator $G$ acting on the function $H(a_n)$, which also means that
the near field operator can be factorized as 
$$
N = GH.
$$
Now the inverse scattering problem can be stated as follows.\\
{\bf Inverse problem:} find the  support $\overline{D}$ of the periodic 
contrast $Q$ given near field operator $N$.



\subsection{The adjoint operator $H^*$}

We solve the inverse problem using the factorization method. Factorizing the near field operator is one of the 
important steps of this method. Before doing that, in the next lemma, 
we find the adjoint  $H^*$ of operator $H$ in~\eqref{eq:IncidentField2}. 
\begin{lemma}
\label{th:H*}
For $f \in [L^2(D)]^2$, the adjoint $H^*:[L^2(D)]^2 \to [\ell^2(\Z)]^2$ of  operator $H$ in~\eqref{eq:IncidentField2}
satisfies
\begin{equation}
\label{eq:H*}
 (H^*f)_n = 4\pi
\left( \begin{array}{ccc}
\widetilde{w}^+_n  &\widetilde{w}^+_n  \\
\widetilde{w}^-_n & -\widetilde{w}^-_n
\end{array} \right)
\left( \begin{array}{ccc}
\widehat{u}^+_n  \\
\widehat{u}^-_n
\end{array} \right), \quad n \in \Z,
\end{equation}
where
\begin{equation*}
\widetilde{w}^+_n =  \left \{ \begin{array}{ll}
 \text{e}^{-\i\beta_n h},  & k^2 > \alpha_n^2, \\
\i,   & k^2 < \alpha_n^2,\\
          \end{array} \right. \quad
\widetilde{w}^-_n=  \left \{ \begin{array}{ll}
 \text{e}^{-\i\beta_n h},  & k^2 > \alpha_n^2, \\
1,   & k^2 < \alpha_n^2,\\
          \end{array} \right. 
\end{equation*}
and $\widehat{u}^{\pm}_n$ are the Rayleigh sequences of  radiating solution $u \in H^1_{\alpha}(\Omega_h)$ to
\begin{equation}
 \label{eq:eqH*}
\mathcal{B}(u,v;I)   = \int_{D}  (\Re Q)^{1/2}f \cdot \nabla \overline{v}\d{x}, \quad \text{for all } v\in H^1_{\alpha}(\Omega_h).
\end{equation}
\end{lemma}
\begin{proof}
First, the problem~\eqref{eq:eqH*} is uniquely solvable for all wave numbers $k>0$ (see~\cite{Bonne1994}).
Now we compute
\begin{align*}
 &\int_{D} H(a_n) \cdot \overline{f} \d x = \sum\limits_{n \in \Z} \left [ \frac{a^{+}_n}{\beta_nw^+_n} 
\int_{D}(\Re Q)^{1/2} \nabla\varphi^+_n \cdot \overline{f}\d x + \frac{a^-_n}{\beta_nw^-_n} \int_{D}(\Re Q)^{1/2}  \nabla\varphi^-_n \cdot \overline{f}\d x \right ] \\ 
&= \Bigg\langle(a_n), \left( \int_{D} (\Re Q)^{1/2} \nabla \left(\ol{\frac{\varphi_n^{+}}{\beta_nw^+_n}}
\right) \cdot f  \d x, \int_{D}  (\Re Q)^{1/2} \nabla \left(\ol{\frac{\varphi_n^{-}}{\beta_nw^-_n} } \right) \cdot f \d x \right)^\top \Bigg\rangle_{[\ell^2(\Z)]^2}.
\end{align*}
Setting $g_n = \varphi_n^+/(\beta_nw^+_n) $ we now compute $ \int_{D}  (\Re Q)^{1/2}\nabla \ol{g_n}\cdot f \d x$ in the inner product above.

Let  $u$ be the radiating solution to~\eqref{eq:eqH*}. Since $(\Re Q)^{1/2}$ is symmetric,
$$ \int_{D}  (\Re Q)^{1/2}\nabla \ol{g_n}\cdot f \d x = \int_{D}  (\Re Q)^{1/2} f \cdot \nabla \ol{g_n} \d x.$$ 
Letting $v = g_n$ in~\eqref{eq:eqH*}, using Green's theorems and the fact that $\Delta g_n +k^2 g_n = 0$ we obtain
\begin{align}
\label{comput1}
\int_{D}(\Re Q)^{1/2} f \cdot  \nabla \ol{g_n}\d{x} =  \int_{\Gamma_h}  u\partial_{x_2}{\ol{g_n}} - T^+(u) \ol{g_n} \d s  - \int_{\Gamma_{-h}} u\partial_{x_2}{\ol{g_n}} + T^-(u) \ol{g_n} \d s.
\end{align}
From a straightforward calculation we further have
\begin{align*}
\overline{g_n}|_{\Gamma_h} &= \overline{g_n}|_{\Gamma_{-h}} = -\frac{1}{\beta_nw^+_n}(\text{e}^{\i \beta_nh} +  \text{e}^{-\i \beta_nh})e^{-\i\alpha_n x_1}, \\
 \partial_{x_2} \overline{g_n}|_{\Gamma_h} &= -\partial_{x_2} \overline{g_n}|_{\Gamma_{-h}} = -\frac{\i}{w^+_n}(\text{e}^{\i \beta_nh}- \text{e}^{-\i \beta_nh})e^{-\i\alpha_n x_1}. 
\end{align*}
Substituting these equations and the  radiation condition 
$u|_{\Gamma_{\pm h}} = \sum_{j \in \Z}\widehat{u}^\pm_j \text{e}^{\i \alpha_j x_1}$ in~\eqref{comput1}, and  
doing some calculations we obtain
\begin{align*}
\int_{D} (\Re Q)^{1/2} f \cdot \nabla \overline{g_n}   \d{x} &= \frac{2\i \widehat{u}^+_n }{w^+_n} \int_{\Gamma_h} \text{e}^{-\i \beta_n h} \d s + \frac{2\i \widehat{u}^-_n}{w^+_n }  \int_{\Gamma_{-h}} \text{e}^{-\i \beta_n h} \d s  \\
&= \left \{ \begin{array}{ll}
 4\pi \text{e}^{-\i\beta_nh}(\widehat{u}^+_n+\widehat{u}^-_n) ,  & k^2 > \alpha_n^2 \\
 4 \pi \i (\widehat{u}^+_n+\widehat{u}^-_n),   & k^2 < \alpha_n^2
          \end{array} \right. 
=  4 \pi\widetilde{w}^+_n(\widehat{u}^+_n+\widehat{u}^-_n).
\end{align*}
Similarly we obtain 
\begin{align*}
 \int_{D}  (\Re Q)^{1/2} \nabla \left(\ol{\frac{\varphi_n^{-}}{\beta_nw^-_n} } \right) \cdot f \d x=  4 \pi\widetilde{w}^-_n(\widehat{u}^+_n-\widehat{u}^-_n).
\end{align*}
This shows that $H^*$ is given by~\eqref{eq:H*}.
\end{proof}
We need the following operators in our analysis.
Let $W: [\ell^2(\Z)]^2 \to [\ell^2(\Z)]^2$  defined by
\begin{equation}
\label{eq:W}
\left(W \left( \begin{array}{ccc}
 a_n  \\
 b_n
\end{array} \right)\right)_n =
4\pi
\left( \begin{array}{ccccc}
 \widetilde{w}^{+}_n &   \widetilde{w}^{-}_n  \\
 \widetilde{w}^{+}_n & - \widetilde{w}^{-}_n
\end{array} \right)\left( \begin{array}{ccc}
a_n \\
b_n
\end{array} \right), \quad n \in \Z,
\end{equation}
and $E: [L^2(D)]^2 \to [\ell^2(\Z)]^2$  defined by
\begin{equation}
\label{E}
 (Ef)_n = 
\left( \begin{array}{ccc}
\widehat{u}^+_n  \\
\widehat{u}^-_n
\end{array} \right), \quad n \in \Z,
\end{equation}
where $\widehat{u}^{\pm}_n$ are the Rayleigh sequences of the  radiating solution $u \in H^1_{\alpha}(\Omega_h)$ to
$$\mathcal{B}(u,v;I)   = \int_{D}  (\Re Q)^{1/2}f \cdot \nabla \overline{v}\d{x} \quad \text{ for all } \quad v\in H^1_{\alpha}(\Omega_h).$$
It is easy to see that these are linear bounded operators and that $H^* = WE$. 
Moreover, $W$ has a bounded inverse because
\begin{align*}
\det 
\left( \begin{array}{ccccc}
 \widetilde{w}^{+}_n & \widetilde{w}^{-}_n  \\
\widetilde{w}^{+}_n & -\widetilde{w}^{-}_n
\end{array} \right) = -2\widetilde{w}^{+}_n \widetilde{w}^{-}_n \neq 0 \quad \text{ for all } n \in \Z.
\end{align*}

Next, we show that  the range  of the adjoint operator $H^*$, denoted by $\Rg(H^*)$, can characterize $D$. 
To this end, we first need the quasi-periodic Green function of the direct problem~\eqref{eq:HmodeEquation1}--\eqref{eq:radiationCondition}
(see e.g.~\cite{Lechl2014})
\begin{equation}
 \label{eq:GreenForm1}
 \mathcal{G}(x,z) = \frac{\i}{4\pi}\sum_{n \in\Z} \frac{1}{\beta_n}\text{e}^{\i\alpha_n(x_1-z_1) + \i\beta_n|x_2-z_2|}, 
\quad x,z\in \Omega,\, x_2\neq z_2.
\end{equation}
It is easy to check that, for a fixed $z$,  the Rayleigh sequences of $\mathcal{G}(x,z)$ are given by
\begin{equation}
 \label{eq:TestGreen}
  r^\pm_n(z) = \frac{\i}{4\pi\beta_n}\text{e}^{-\i\alpha_n z_1 \pm \i\beta_n(z_2 \mp h)}.
 \end{equation}
 We have the following characterization of $D$.
\begin{lemma}
\label{characterization}
A point $z$ in $\Omega$ belongs to $D$ if and only if 
$$
W\left( \begin{array}{ccc}
 r^+_n(z)  \\
 r^-_n(z)
\end{array} \right) \in \Rg(H^*).$$
\end{lemma}
\begin{proof}
For $z \in D$, let $\rho>0$ such that the ball $B(z,\rho)$ belongs to  $D$. Consider a 
smooth function $\xi$ which is $2\pi$-periodic in $x_1$  and satisfies
$\xi(x) = 0$ in $B(z,\rho/2)$ and $\xi(x) = 1$ for $|x-z| \geq \rho$.  Then the function
$$
\Phi(x) := \frac{1}{k^2}\Delta(\xi(x) \mathcal{G}(x,z))
$$
is an quasi-periodic smooth function and $\Phi(x) = - \mathcal{G}(x,z)$ for $|x-z| \geq \rho$. 
For $v \in H^1_\alpha(\Omega_h)$, substituting  $\Phi =  \frac{1}{k^2}\Delta(\xi \mathcal{G}(\cdot,z))$
in the zero-order term of $\mathcal{B}(\Phi,v;I)$ and using Green's identities we obtain
\begin{align*}
&\mathcal{B}(\Phi,v;I) = \int_{\Omega_h} \nabla \Phi  \cdot \nabla\ol{v} - \Delta(\xi \mathcal{G}(\cdot,z)) \ol{v} \d x 
- \int_{\Gamma_ h} T^+(\Phi)\ol{v}\d{s}
    - \int_{\Gamma_{- h}} T^-(\Phi)\ol{v}\d{s} \\ 
   & =  \int_{\Omega_h} [\nabla \Phi + \nabla (\xi \mathcal{G}(\cdot,z))] \cdot \nabla \ol{v}\d x
   + \int_{\Gamma_ h}[- \partial_{x_2}\mathcal{G}(\cdot,z) + T^+(\mathcal{G}(\cdot,z) ) ]\ol{v}\d{s}  \\ 
   & \quad + \int_{\Gamma_ {-h}}[\partial_{x_2}\mathcal{G}(\cdot,z) + T^-(\mathcal{G}(\cdot,z) ) ]\ol{v}\d{s}
\end{align*}
The boundary terms are zero because of the definition of $T^{\pm}$. 
Let 
$$f = (\Re Q)^{-1/2} [\nabla \Phi + \nabla (\xi \mathcal{G}(\cdot,z))] \in [L^2(D)]^2.$$  
Then 
$f$ is supported in $D$  since $ \Phi = - \mathcal{G}(\cdot,z)$  in $\Omega\setminus \ol{D}$. 
Therefore, we have proven
that 
$$
\mathcal{B}(\Phi,v;I) =  \int_D( \Re Q)^{1/2} f \cdot \nabla \ol{v} \d x, \quad \text{for all } v \in H^1_\alpha(\Omega_h).
$$
This means that there exists $f \in [L^2(D)]^2$ such that  
$$
(Ef)_n =  (\widehat{\Phi}^{+}_n(\cdot,z),\widehat{\Phi}^{+}_n(\cdot,z))^\top = (r^+_n(z),r^-_n(z))^\top.
$$
Together with $H^* = WE$, this implies that
$$ W(r^+_n(z),r^-_n(z))^\top \in \Rg(H^*).$$

 Now suppose  that 
$z \notin D$ and $(r^+_n(z),r^-_n(z))^\top \in \Rg(E)$.
Then there exists $u \in H^1_{\alpha}(\Omega_h)$ solving problem~\eqref{eq:eqH*} for 
some $f\in [L^2(D)]^2$ in the right hand side, and $\widehat{u}^\pm_n = r^\pm_n(z)$ 
for all $n\in \Z$. This implies that $u = \mathcal{G}(\cdot,z)$
in $\Omega\setminus \overline{(-h, h)}$. Since $u$ and $\mathcal{G}(\cdot,z)$ 
are respectively analytic functions
in $\Omega\setminus D$ and $\Omega\setminus\{z\}$,  the analytic continuation 
implies that $u = \mathcal{G}(\cdot,z)$ in  $\Omega\setminus(D\cup\{z\})$. 
However, it is well known  that 
$\mathcal{G}(\cdot,z)$ is singular at $z$ which leads to a contradiction 
since $u \in H^1(O)$ for some neighborhood $O$ of $z$ but $\mathcal{G}(\cdot,z) \notin H^1(O)$ 
due to the singularity at $z$.
\end{proof}

We  can't find $D$ yet since $H^*$ is defined on $[L^2(D)]^2$. 
One of the most important steps of the factorization method is to connect $\Rg(H^*)$ 
to something related to the near field operator $N$ that is given. This is the content of 
the next section.


\subsection{Shape reconstruction by the factorization method}
The following operator is crucial in the factorization method.
  Let  $T:[L^2(D)]^2 \to [L^2(D)]^2$  
 defined by 
\begin{align}
\label{eqT}
Tf = (\Re Q)^{-1/2} Q ((\Re Q)^{-1/2}f + \nabla u), 
\end{align}
where $u \in H^1_{\alpha}(\Omega_h)$ is the 
solution to~\eqref{eq:VariationalInverse2}. It is not difficult to see that this is a linear bounded operator.
Now we factorize  the operator $N$ in the following theorem. 
\begin{lemma}
\label{factorization}
Suppose that the Assumption~\ref{th:assumption1} holds true. Then  near field operator $N$ can be factorized as
\[
 WN = H^*TH.
\]
\end{lemma}
\begin{proof}
In  the definition of the operator $G$ in~\eqref{G} we observe that the variational problem~\eqref{eq:VariationalInverse2}
can be written as $${\displaystyle \mathcal{B}(u,v;I) =  \int_D (\Re Q)^{1/2} (\Re Q)^{-1/2} Q((\Re Q)^{-1/2}f + \nabla u) \cdot \nabla \ol{v}\d x}.$$
This means that for all $f \in [L^2(D)]^2$
$$
G f= ETf.
$$
Thanks to the facts that $N = GH$ and $H^* = WE$ we have
\begin{eqnarray*}
WN = WGH = WETH = H^*TH,
\end{eqnarray*}
which completes the proof.
\end{proof}
Let   $T_0:[L^2(D)]^2 \to [L^2(D)]^2$  be defined by $$T_0f = (\Re Q)^{-1/2} Q ((\Re Q)^{-1/2}f + \nabla \widetilde{u})$$ where 
$\widetilde{u}\in H^1_{\alpha}(\Omega_h)$ solves~\eqref{eq:VariationalInverse2} for $k=\i$. 
We have the following analytical properties of the operators in the factorization obtained above.
\begin{lemma}
\label{properties}
Suppose that Assumption~\ref{th:assumption1} holds true. 
Then operators $H$ and $T$ satisfy the following:\\
(a) $H$ is compact and injective.\\
(b) $T$ is injective, $\langle\Im T f, f \rangle \leq 0$ for all $f \in [L^2(D)]^2$, and 
$\langle\Im T f, f \rangle < 0$  for all $f \neq 0$ in $[L^2(D)]^2$. \\
(c)   $T-T_0$ is compact and $\Re(T_0)$ is coercive in $[L^2(D)]^2$. 
\end{lemma}
\begin{proof}
The proofs for these properties  can be done following their analogues  of the half-space case~\cite{Nguye2014} and therefore 
are omitted here. 
\end{proof}
From the range identity theorem~\cite{Kirsc2008}, these analytical properties and the factorization in Lemma~\ref{factorization} 
allow us to obtain that $\Rg(H^*) = \Rg (WN)^{1/2}_{\sharp}$ where
$$
(WN)_{\sharp} = |\Re WN|-\Im WN
$$
is a positive definite operator. Therefore, from Lemma~\ref{characterization} we
now have a necessary and sufficient characterization of $D$ in terms of $\Rg (WN)^{1/2}_{\sharp}$. 
Since $(WN)_{\sharp}$ is a compact and self adjoint operator, we can exploit 
its eigensystem for imaging of $D$ from the near field data. This is the content of
the following theorem.

\begin{theorem}
Suppose that Assumption~\ref{th:assumption1} holds true.
For $j\in \Z$, denote by $(\lambda_j, \psi_{n,j})_{j \in \N}$ an orthonormal 
eigensystem of $(WN)_{\sharp}$. Then a point $z \in \Omega$ 
belongs to $\overline{D}$ if and only if
\begin{equation}
 \label{eq:criteria}
\sum\limits_{j=1}^{\infty} \frac{|\langle r_n(z), \psi_{n,j} \rangle_{[\ell^2(\Z)]^2}|^2}{\lambda_j} < \infty.
\end{equation} 
\end{theorem}

\begin{proof}
The proof is similar to that of the half space case~\cite{Nguye2014}.
\end{proof}


\section{The transmission eigenvalue problem  }
In this section, we derive and study the corresponding transmission eigenvalue problem for the scattering by an anisotropic periodic layer. In general, {\color{black} the real  eigenvalues} can be recovered from the scattering data and can be used to determine the material properties of the anisotropic periodic layer. {\color{black}We wish to prove the existence of infinity many real transmission eigenvalues}. See for e.g. \cite{Cakon2015} for the estimation of the effective material properties for a highly oscillatory media and \cite{Harri2014} for the recovery of the transmission eigenvalues for an anisotropic media from the scattering data. {\color{black}Since, in general it is known that absorbing materials do not have real transmission eigenvalues we will assume in this section that the scatterer is non-absorbing (i.e. $\textrm{Im } Q=0$) for the study of transmission eigenvalues.}

We now derive our transmission eigenvalue problem which corresponds to the wave numbers $k$ for which the scattered field vanishes away from the object for some non-trivial quasi-periodic incident field. This means that there is a quasi-periodic incident field $u_{\text{in}} \neq 0$ that is a solution to the Helmholtz equation such that the scattered field $u_{\text{sc}} = 0$ for all $|x_2|>h$ by Holmgren's theorem and the Rayleigh expansion condition \eqref{eq:radiationCondition}. By appealing to Holmgren's theorem again we obtain that $u_{\text{sc}} = 0$ for all $x \in \Omega_h \setminus \overline{D}$. Now assuming that 
$$\text{sup}\big\{ |x_1| \, : \,  (x_1,x_2)^\top \in \supp(Q) \cap \Omega_h \big\}<\pi$$
then we have that $\partial \Omega_h \cap \partial D$ is empty. Therefore, we have that $w = u_{\text{in}} + u_{\text{sc}}$ and $v =u_{\text{in}}$ are in $H^1(D)$ satisfying 
\begin{eqnarray}
\div(A \grad w) +k^2 w=0 \quad && \textrm{ and } \qquad  \Delta v + k^2 v=0 \quad \textrm{ in } \,\,  D\label{te-prob1}\\
w=v \quad && \textrm{ and } \qquad \frac{\partial w}{\partial \nu_{A}}=\frac{\partial v}{\partial \nu} \quad \textrm{ on } \,\,\partial D \label{te-prob2}
\end{eqnarray}
where for a generic function $\partial \varphi/\partial \nu_A=\nu \cdot A  \nabla \varphi$. Notice, that the transmission eigenvalue problem \eqref{te-prob1}--\eqref{te-prob2} has already been studied (see for e.g. \cite{Cakon2016}). Therefore, we now assume that 
$$\text{sup}\big\{ |x_1| \, : \,  (x_1,x_2)^\top \in \supp(Q)\cap \Omega_h \big\}=\pi$$
which implies that 
\begin{eqnarray}
D= \big\{ (x_1,x_2)^\top \in \R^2 \, : \, -\pi < x_1 < \pi \,\, \text{ and } \,\, f_{-} (x_1) < x_2 < f_{+}(x_1) \big\}  \label{D-rep}
\end{eqnarray}
where $f_{\pm} \in C^{0,1}[-\pi,\pi]$. We further denote
$$
\Gamma_\pm = \{(x_1,x_2)^\top \in \ol{D}: x_2 = f_\pm(x_1)\}.
$$
Now if  $u_{\text{sc}} = 0$ for any $x_2 > f_{+}(x_1)$ and $x_2 < f_{-}(x_1)$ for all $x_1 \in (-\pi,\pi)$,  then we have that $w = u_{\text{in}} + u_{\text{sc}}$ and $v =u_{\text{in}}$ are in $H^1_\alpha (D)$ satisfying
\begin{eqnarray}
\div( A \grad w) +k^2  w=0 \quad && \textrm{ and } \qquad  \Delta v + k^2 v=0 \quad \textrm{ in } \,\,  D \label{te-prob3}\\
w=v \quad && \textrm{ and } \qquad \frac{\partial w}{\partial \nu_{A}}=\frac{\partial v}{\partial \nu} \quad \textrm{ on  } \,\,\Gamma_\pm.  \label{te-prob4}
\end{eqnarray}
Notice, that here to completely formulate the eigenvalue problem we must enforce the quasi-periodic boundary condition  which is not needed for the problem in the previous case. 

We say that $k$ is a {\it  transmission eigenvalue} provided that there is a non-trivial solution $(w,v) \in H^1_\alpha (D)^2$ satisfying \eqref{te-prob3}--\eqref{te-prob4}. The transmission eigenvalue problem \eqref{te-prob3}--\eqref{te-prob4} is new since the case where the eigenfunctions are quasi-periodic has not been studied. Now following the analysis in \cite{Cakon2010,Harri2014} we will prove the existence of transmission eigenvalues. Notice that a 4-th order formulation is not used as is done for the standard transmission eigenvalue problem (see for e.g. \cite{Cakon2016}). To this end, we will need the following Poincar\'{e}  inequality result for $H^1_\alpha (D)$.

\begin{lemma}
For all $u \in  H^1_\alpha (D)$ we have that $\|u\|^2_{L^2(D)} \leq C_\alpha \|\nabla u\|^2_{L^2(D)}$ provided $\alpha \notin \Z$ where 
 $C_\alpha$ is a positive constant that is independent of $k$. 
\end{lemma}
\begin{proof}
Assume on the contrary that $H^1_\alpha (D)$ does not satisfy a Poincar\'{e} inequality. This implies that we can find a sequence $u_n$ such that $\|u_n \|^2_{L^2(D)} = 1$ for all $n \in \N$ and $\|\nabla u_n \|^2_{L^2(D)} \to 0$ as $n \to \infty$. By Rellich’s  compact embedding we can conclude that $u_n$ (up to a subsequence) converges weakly in $H^1(D)$ to $u$ such that the weak limit satisfies $\|u \|^2_{L^2(D)} = 1$ and $\|\nabla u \|^2_{L^2(D)} = 0$. We have that $u$ is a non-zero constant and quasi-periodic which implies that $1=\text{e}^{2\pi \i \alpha}$ which can not be since $\alpha \notin \Z$. Therefore, we have proven that there is a constant $C_\alpha > 0$ such that $\|u\|^2_{L^2(D)} \leq C_\alpha \|\nabla u\|^2_{L^2(D)}$ for all $u \in  H^1_\alpha (D)$. 
%
\end{proof}

In order to insure that the space $H^1_{\alpha}(D)$ satisfies the above Poincar\'{e} inequality we will  now assume that $\alpha \notin \Z$ for the rest of the section. We now reformulate the quasi-periodic transmission eigenvalue problem \eqref{te-prob3}--\eqref{te-prob4} as a problem for $u=v-w \in H^1_{0,\alpha}(D)$ where the Hilbert space 
$$ H^1_{0,\alpha}(D)=\big\{ u \in H_\alpha^1(D) \, : \, u = 0 \quad \textrm{on } \Gamma_\pm  \big\}$$
equipped with the $H^1(D)$ norm. By subtracting the equations and boundary conditions in \eqref{te-prob3}--\eqref{te-prob4} for the eigenfunctions $w$ and $v$ we obtain 
 \begin{eqnarray}
\div( A \grad u) +k^2u=\div( Q \grad v) &\quad \textrm{in }& D \label{te-prob5}\\
\frac{\partial u}{\partial \nu_A}=\nu \cdot Q\grad v &\quad \textrm{on  }& \Gamma_\pm. \label{te-prob6}
\end{eqnarray}
In order to completely reformulate the problem for the difference $u$ we must show that \eqref{te-prob5}--\eqref{te-prob6} defines a bounded linear mapping $u \mapsto v$ from $ H^1_{0,\alpha}(D) \longmapsto H^1_\alpha (D)$. Notice that this implies that \eqref{te-prob3}--\eqref{te-prob4} and \eqref{te-prob5}--\eqref{te-prob6} are equivalent provided that $v$ satisfies~\eqref{te-prob3} by taking $w=v-u$ since $H^1_{0,\alpha}(D) \subset H^1_\alpha (D)$. The variational formulation of  \eqref{te-prob5}--\eqref{te-prob6} is given by 
\begin{eqnarray}
\int_{D} Q \grad v \cdot \grad \overline{\varphi} \, \d x =  \int_{D} A \grad u \cdot \grad \overline{\varphi} - k^2u\overline{\varphi} \, \d x \quad \textrm{ for all } \, \, \varphi \in H^1_\alpha (D). \label{utov}
\end{eqnarray}
Due to the Poincar\'{e} inequality and the fact that $Q$ is a uniformly positive definite matrix with bounded entries give that \eqref{utov} is well-posed by the Lax-Milgram theorem. Next, we define the operator $\Lk$ that maps $ H^1_{0,\alpha}(D)$ into itself via the Riesz representation theorem such that 
\begin{equation}
\label{dl}
\big(\Lk u, \varphi\big)_{H^1(D)}=\int_{D} \grad v_u \cdot \grad \overline{\varphi} -k^2 v_u  \overline{\varphi}\, \d x  \quad \textrm{ for all } \, \, \varphi \in H^1_{0,\alpha} (D).
\end{equation}
where $v_u$ is the unique solution to \eqref{utov}. This operator  $\Lk$ is a key ingredient to the study of the transmission eigenvalue problem.  Note that the right hand side of~\eqref{dl} is designed to take into account the fact that $v_u$ satisfies~\eqref{te-prob3}. It is obvious that $\Lk$ depends continuously on $k$. More importantly, we have that if  $u$ is in the kernel of $\Lk$ for some $k>0$, then $w=v_u-u$ and $v_u$ are quasi-periodic transmission eigenfunctions with  transmission eigenvalue  $k$. Vice versa, if $w$ and $v$ are quasi-periodic transmission eigenfunctions with eigenvalue $k$, then $u = v-w$  belongs
to the kernel of $\Lk$.

In order to prove the existence of the transmission eigenvalues $k$ we need to determine some properties of the operator $\Lk$. To this end, we will denote
 $v_j$ to be the unique solution to \eqref{utov} for a given $u_j$ and $w_j = u_j-v_j$ for $j=1,2$. Then similar calculations as in \cite{Cakon2010} gives that 
\begin{eqnarray}
\big(\Lk u_1, u_2\big)_{H^1(D)} = \int_{D} \grad u_1 \cdot \grad \overline{u}_2 -k^2 u_1  \overline{u}_2 \, \d x + \int_{D} Q \grad w_1 \cdot \grad \overline{w}_2 \, \d x. \label{lproperties}
\end{eqnarray}
Notice, that since $Q$ is a real symmetric matrix we have that the sesquilinear form in \eqref{lproperties} is Hermitian giving that $\Lk$ is a selfadjoint operator. Now, taking $k=0$ we obtain that 
$$\big(\mathbb{L}_0 u, u\big)_{H^1(D)} = \int_{D} | \grad u |^2 \, \d x + \int_{D} Q  \grad w\cdot\grad \ol{w} \, \d x$$
which gives that $\mathbb{L}_0$ is a coercive operator due to the Poincar\'{e} inequality and the fact that $Q$ is a positive definite matrix. We now show that the operator $\Lk-\mathbb{L}_0$ is compact. Indeed, let the sequence $u_n$ in $H^1_{0,\alpha}(D)$ weakly converge to zero as $n \to \infty$. Therefore, we have that the sequence of solutions to \eqref{utov} denoted $v_{n,k}$ in $H^1_{0,\alpha}(D)$ (where we explicitly denote the dependance on  $k$) weakly converges to zero as $n \to \infty$ for all $k \in \R$ by the well-posedness of equation \eqref{utov}. Rellich's compact embedding implies that $u_n$ and $v_{n,k}$  converges to zero in the $L^2(D)$ norm. By subtracting equation \eqref{utov} for $k \neq 0$ and $k=0$ gives that 
$$\int_{D }Q \grad (v_{n,k}-v_{n,0}) \cdot \grad \overline{\varphi} \, \d x= -k^2 \int_{D} u_n \overline{\varphi}  \, \d x  \, \, \, \, \text{ for any }  \, \, \varphi \in H_{\alpha}^1(D)$$
which implies that 
$v_{n,k}-v_{n,0}$ converges to zero in the $H^1(D)$ norm by letting $\varphi = v_{n,k}-v_{n,0}$ and appealing to fact that $Q$ is uniformly positive definite. We now have that 
\begin{eqnarray*}
\Big( (\Lk- \mathbb{L}_0)u_n, \varphi \Big)_{H^1(D)}=\int_{D} \grad (v_{n,k} - v_{n,0}) \cdot \grad \overline{\varphi} -k^2  v_{n,k} \overline{\varphi}  \, \d x  \end{eqnarray*}
 and by the Cauchy-Schwartz inequality

$$\Big\| (\Lk -\mathbb{L}_0)u_n \Big\|_{H^1(D)} \leq \Big( || v_{n,k} - v_{n,0} ||_{H^1(D)} + k^2 || v_{n,k} ||_{L^2(D)} \Big) \to 0 \quad \text{as} \quad n \to \infty.$$
This gives that $\Lk -\mathbb{L}_0$ is compact. From the above analysis we have the following result. 

\begin{lemma}\label{lemmap}
For any $k \in \R$ the operator $\Lk : H^1_{0,\alpha}(D)  \longmapsto H^1_{0,\alpha}(D)$ satisfies:
\begin{enumerate}
\item $\Lk$ is self-adjoint
\item $\mathbb{L}_0$ is coercive
\item $\Lk- \mathbb{L}_0$ is a compact.
\end{enumerate}
\end{lemma}

By appealing to the theory developed in \cite{Cakon2010} in order to prove the existence of  transmission eigenvalues we now need to show that $\Lk$ is positive on $H^1_{0,\alpha}(D)$ for some $k_{m}$ and is non-positive on a $M$--dimensional subspace of $H^1_{0,\alpha}(D)$ for some $k_{M}$. This would imply that  there are $M$  transmission eigenvalues by the arguments in Section 2.1 of \cite{Cakon2010}.  

\begin{theorem}\label{eig-exist}
There exists infinitely many  transmission eigenvalues.
\end{theorem}
\begin{proof}
We begin by showing that for all $k$ sufficiently small the operator $\Lk$ is positive. To this end, notice that by \eqref{lproperties} and the fact that $Q$ is a positive definite matrix we have that 
$$\big(\mathbb{L}_k u, u\big)_{H^1(D)} \geq \int_{D} | \grad u |^2 -k^2 | u |^2\, \d x.$$
By appealing to the Poincar\'{e} inequality we obtain the estimate 
$$\big(\mathbb{L}_k u, u\big)_{H^1(D)} \geq  \left( 1- k^2 C_\alpha \right) \int_{D} | \grad u |^2 \, \d x.$$
We can then conclude that $\Lk$ is positive on $H^1_{0,\alpha}(D)$ for all $k^2 < 1/C_\alpha$. 

Now the goal is to prove that for some subspace of $H^1_{0,\alpha}(D)$ and value $k$ that $\Lk$ is non-positive. Therefore, we let $\overline{B} \subset D$ be a ball of radius $\eps$ centering at some point $x \in D$. Define $k_\eps>0$ to be the smallest transmission eigenvalue of 
\begin{eqnarray}
A_{\text{min}} \Delta w_\eps +k_\eps^2 w_\eps=0 \quad && \textrm{ and } \qquad  \Delta v_\eps + k_\eps^2 v_\eps=0 \quad \textrm{ in } \,\,  B \label{stuff1}\\
w_\eps=v_\eps \quad && \textrm{ and } \qquad \frac{\partial w_\eps}{\partial \nu_{A}}=\frac{\partial v_\eps}{\partial \nu} \quad \textrm{ on } \,\,\partial B \label{stuff2}
\end{eqnarray}
where 
$$ \inf_{ x \in D} \inf_{|\xi|=1} \overline{\xi} \cdot A(x) \xi =A_{\text{min}} \quad \text{ which gives that } \quad  \inf_{ x \in D} \inf_{|\xi|=1} \overline{\xi} \cdot Q(x) \xi =A_{\text{min}} - 1 .$$
We can define $u_\eps = v_\eps - w_\eps$ in $H^1_0(B)$ and its extension by zero to all of $D$ by $ u$ in $H^1_0(D)$. Now let $ v$ in $H^1_\alpha (D)$ be the solution to \eqref{utov} with $k_\eps$ and $ u$ and $ w =  v - u$. Using Green's theorem and some simple calculations give that (see for e.g. \cite{Cakon2010})
$$\int_{B} \big(A_{\text{min}} - 1\big) \grad w_\eps \cdot \grad \overline{\varphi} \,\d x = \int_{B}  \grad u_\eps \cdot \grad \overline{\varphi}  - k_\eps^2 u_\eps \overline{\varphi}\, \d x$$
and 
$$\int_{D} Q \grad  w \cdot \grad \overline{\varphi} \,\d x = \int_{D}  \grad  u \cdot \grad \overline{\varphi}  - k_\eps^2  u \overline{\varphi}\, \d x \quad \text{ for all } \quad \varphi \in H^1_\alpha (D).$$
Now, notice that by the definition of $ u$ 
\begin{eqnarray*}
\int_{D} Q \grad  w \cdot \grad \ol\varphi \,\d x  = \int_{D}  \grad  u \cdot \grad \overline{\varphi}  - k_\eps^2  u \overline{\varphi}\, \d x &=& \int_{B}  \grad  u_\eps \cdot \grad \overline{\varphi}  - k_\eps^2 u_\eps \overline{\varphi}\, \d x\\
&=& \int_{B} \big(A_{\text{min}} - 1\big) \grad w_\eps \cdot \grad \overline{\varphi}  \, \d x 
\end{eqnarray*}
Letting $\varphi =  w$ and estimating gives 
\begin{eqnarray*}
\int_{D} Q \grad  w \cdot \grad \ol w\,\d x &=& \int_{B} \big(A_{\text{min}} - 1\big) \grad w_\eps \cdot \grad \overline{ w}  \, \d x \\
&\leq& \left( \int_{B} \big(A_{\text{min}} - 1\big) |\grad w_\eps|^2 \, \d x \right)^{1/2} \left( \int_{B} \big(A_{\text{min}} - 1\big) |\grad  w|^2 \, \d x \right)^{1/2}\\
&\leq& \left( \int_{B} \big(A_{\text{min}} - 1\big) |\grad w_\eps|^2 \, \d x \right)^{1/2} \left( \int_{D} Q  \grad  w\cdot \grad \ol w \, \d x \right)^{1/2}.
\end{eqnarray*}
Therefore, we obtain that 
\begin{eqnarray*}
\big(\mathbb{L}_{k_\eps}  u,  u \big)_{H^1(D)} &=& \int_{D} | \grad  u |^2 -k_\eps^2 |  u |^2 \, \d x + \int_{D} Q  \grad  w\cdot \grad \ol w \, \d x\\
&\leq& \int_{B} | \grad  u_\eps |^2 -k_\eps^2 | u_\eps |^2 \, \d x + \int_{B} \big(A_{\text{min}} - 1\big)   |\grad  w_\eps|^2 \, \d x =0.
\end{eqnarray*}
This implies that $\mathbb{L}_{k_\eps}$ is non-positive on the subspace of $H^1_{0,\alpha}(D)$ which is the span of $u$ proving the existence of a transmission eigenvalue. 

We now wish to construct an infinite dimensional subspace of $H^1_{0,\alpha}(D)$ for which $\mathbb{L}_{k_\eps}$ is non-positive. To this end, we let $B_j$ be the ball centered at $x_j \in D$ with radius $\eps >0.$ Here we define $M_\eps$ to be the supremum of the number of disjoint balls $B_j$ such that $\overline{B_j} \subset D$. Notice that since the coefficient $A_{\text{min}}$ is constant we have $k_\eps$ being the smallest transmission eigenvalue of \eqref{stuff1}--\eqref{stuff2} is the same for each $B_j$. Defining $ u_j$ in $H^1_{0,\alpha}(D)$ for $j=1, \dots, M_\eps$ just as above and we have that since the supports are disjoint $u_j$ is orthogonal to $u_i$ for all $i \neq j$. Therefore, we can conclude that $\text{span}\{ {u}_1 , {u}_2 , \dots,  {u}_{M_\eps}  \}$ is a $M_\eps$--dimensional subspace of $H^1_{0,\alpha}(D)$. Due to the disjoint support of the basis functions we can follow the analysis above to show that $\mathbb{L}_{k_\eps}$ is non-positive for any $u$ in this $M_\eps$--dimensional subspace of $H^1_{0,\alpha}(D)$. Since $M_\eps \to \infty$ as $\eps \to 0$ we can conclude that there are infinitely many transmission eigenvalues. 
\end{proof}

\section{Numerical examples for the shape reconstruction}

In this section, we present some numerical examples examining  the performance of the factorization method 
for  different types of periodic structures and  data perturbed by artificial noise.  We also show  the dependence of the reconstructions 
on the number of the incident fields used.

The synthetic scattering data are generated by  solving the direct  problem with the  spectral Galerkin method 
studied in~\cite{Lechl2014}.  
We solve the direct problem for  incident fields $\varphi^\pm_n$ in~\eqref{eq:IncidentFields} where
 $n=-M,\dots,M$ ($M\in \N$). For each incident field we collect Rayleigh coefficients $\widehat{u}^\pm_j$
 of the corresponding scattered field
 for $j =-M,\dots,M$. The near field operator $N$ is then a $2\times 2$ block matrix. Each block is an $(2M+1) \times
 (2M+1)$ matrix whose $(n,j)$-entry is the $j$th Rayleigh coefficient of the scattered field generated 
 by the $n$th incident field. Two blocks of the block matrix correspond to $\varphi^+_n$ and $\varphi^-_n$ 
 and the other two  are for $\widehat{u}^+$ and $\widehat{u}^-$.
 As in the case of half space~\cite{Nguye2014} using the standard tools of linear algebra we can  easily construct 
 the matrix  $(WN)^{1/2}_\sharp$ and its eigensystem. 
To simulate the case of noisy data we add a noise matrix to data matrix $N$. This noise matrix contains  complex 
random numbers that are uniformly distributed in (0,1). 
To regularize the imaging functional for noisy data we truncate the singular values of 
$(WN)^{1/2}_\sharp$. More precisely, we drop the singular values that are less that 
$5 \times 10^{-4}$. We note that Tikhonov regularization can also be applied 
(as in the half space case~\cite{Nguye2014}) and would give similar results. 

As described above,  $M=10$ in the pictures  means that we use 21 ($2M+1$) incident plane waves and 21 Rayleigh coefficients
of the corresponding scattered fields. It also means that the series in~\eqref{eq:criteria} is truncated with $2(2M+1)$ terms.  
We use the wave number $k = 5.85$ for 
all the  examples. This means that we have
11 propagating modes for the examples and 10 evanescent modes for $M = 10$  and 30 evanescent modes  for $M = 20$. 
The pictures show that the imaging functional based on the factorization method is able to provide reasonable reconstructions
for the shape of several  types of periodic layers.  As in the previous results for the factorization method  for the periodic inverse scattering, 
the evanescent modes are quite important to have reasonable   reconstruction results. 
Here, for all four examples, we have respectively 10 and 30 evanescent modes in the scattering data when $M=10$ and $ M = 30$.

We also observe that the reconstruction results are quite stable with respect to noise
in the data  for the last two examples (Figures 4 and 5). However, the reconstructions for the first two examples are pretty sensitive
to noise, see~\cite{Arens2003a} for a similar situation. The results in Figures 2(d) and 3(d)  are chosen as the best results out of 10 numerical
experiments. We can't see anything reasonable in the worst cases of these numerical reconstructions. The imaging 
functional seems to have more stability in the numerical reconstructions  when the complement of periodic layer in one period is connected but we have no
justification for this behavior.

\section{Summary}

We study the inverse scattering problem for anisotropic periodic layers  and the transmission eigenvalues associated to the problem. 
The Factorization method is investigated as a tool to solve the inverse scattering problem with near field scattering data. This method provides both the unique determination and a fast imaging algorithm for the shape of the periodic scatterer. We present a justification of the Factorization method for the case of absorbing  
materials and some numerical examples to verify its performance. The interior transmission eigenvalue problem for the scattering from anisotropic periodic layers is formulated.  We prove the existence  of infinitely many transmission eigenvalues in the case of non-absorbing materials. 
The periodic scatterer in this paper is assumed to be given by Lipchitz domains and does not have holes. An extension of the results to the case
in which the scatterer has holes or is given by non-Lipchitz domains  is still an open problem.

 \begin{figure} 
\centering
\subfloat[Exact geometry]{\includegraphics[width=6.5cm]{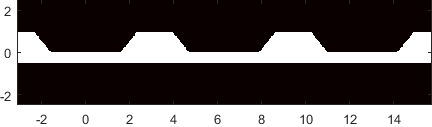}} \hspace{0.5cm}
\subfloat[M = 10]{\includegraphics[width=6.5cm]{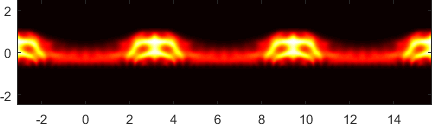}}\\
\subfloat[M = 20]{\includegraphics[width=6.5cm]{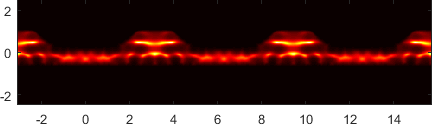}}  \hspace{0.5cm}
\subfloat[M = 20,5$\%$ noise]{\includegraphics[width=6.5cm]{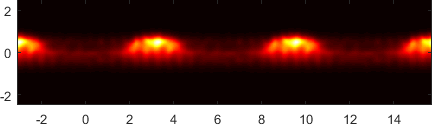}}\\
\caption{Shape reconstruction for the periodic layer of piecewise linear type.  
}
\end{figure}

\begin{figure} 
\centering
\subfloat[Exact geometry]{\includegraphics[width=6.5cm]{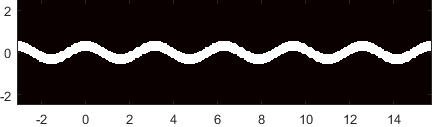}} \hspace{0.5cm}
\subfloat[M = 10]{\includegraphics[width=6.5cm]{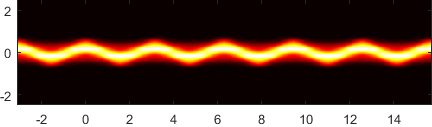}}\\
\subfloat[M = 20]{\includegraphics[width=6.5cm]{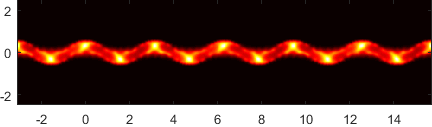}}  \hspace{0.5cm}
\subfloat[M = 20, 5$\%$ noise]{\includegraphics[width=6.5cm]{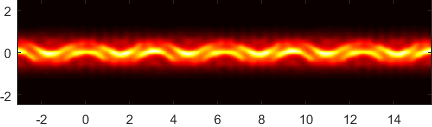}}\\
\caption{Shape reconstruction for the periodic layer of sinusoidal type. }
\end{figure}

 \begin{figure} 
\centering
\subfloat[Exact geometry]{\includegraphics[width=6.5cm]{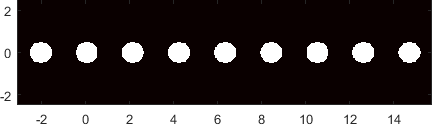}} \hspace{0.5cm}
\subfloat[M = 10]{\includegraphics[width=6.5cm]{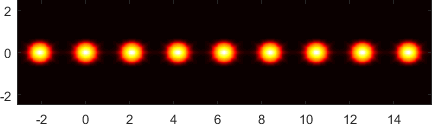}}\\
\subfloat[M = 20]{\includegraphics[width=6.5cm]{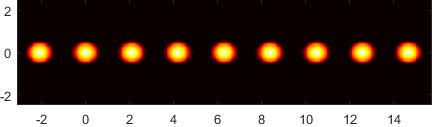}}  \hspace{0.5cm}
\subfloat[M = 20, 5$\%$ noise]{\includegraphics[width=6.5cm]{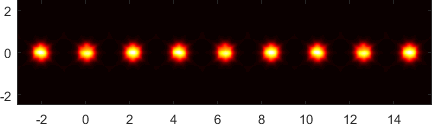}}\\
\caption{Shape reconstruction for the periodic layer of ball type.}
\end{figure}

 \begin{figure} 
\centering
\subfloat[Exact geometry]{\includegraphics[width=6.5cm]{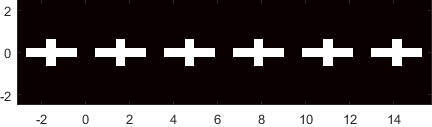}} \hspace{0.5cm}
\subfloat[M = 10]{\includegraphics[width=6.5cm]{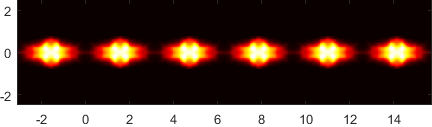}}\\
\subfloat[M = 20]{\includegraphics[width=6.5cm]{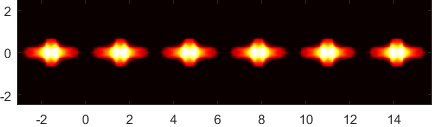}}  \hspace{0.5cm}
\subfloat[M = 20, 5$\%$ noise]{\includegraphics[width=6.5cm]{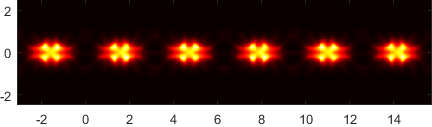}}\\
\caption{Shape reconstruction for the periodic layer of cross type.}
\end{figure}


\end{document}